\documentclass[11pt]{amsart}

\usepackage{graphicx}
\usepackage{amssymb}
\usepackage{amsmath}
\usepackage{enumerate}% http://ctan.org/pkg/enumerate
\usepackage{amsfonts}
\usepackage{amsthm}
\usepackage[all,import]{xy}
\usepackage{xcolor}
\usepackage{url}
\makeatletter
\def\url@leostyle{%
  \@ifundefined{selectfont}{\def\UrlFont{\sf}}{\def\UrlFont{\small\ttfamily}}}
\makeatother
\numberwithin{equation}{section}
\usepackage[colorlinks=true, pdfstartview=FitV, linkcolor=black, 
            citecolor=black, urlcolor=black]{hyperref}
\usepackage{bookmark}
\usepackage[titletoc]{appendix} %For some reason, putting the bookmark and appendix packages after hyperref gets rid of the \endcsname error.

\theoremstyle{definition}
\newtheorem{prop}{Proposition}[section]

\newtheorem{theorem}[prop]{Theorem}
\newtheorem{lemma}[prop]{Lemma}
\newtheorem{corollary}[prop]{Corollary}

\newtheorem{defn}[prop]{Definition}
\newtheorem{example}[prop]{Example}
\newtheorem{remark}[prop]{Remark}

\newcommand{\nc}{\newcommand}
\nc{\DMO}{\DeclareMathOperator}	

\nc{\newnotation}{\nomenclature}
\nc{\wrap}{\cW}
\nc{\Cob}{\mathsf{Cob}}
\nc{\mul}{\mathsf{Mul}}
\nc{\fat}{\mathsf{fat}}
\nc{\cob}{\mathsf{Cob}}
\nc{\coh}{\mathsf{Coh}}
\nc{\idem}{\mathsf{Idem}}
\nc{\sets}{\mathsf{Sets}}
\nc{\near}{\mathsf{near}}
\nc{\sing}{\mathsf{Sing}}
\nc{\symp}{\mathsf{Symp}}
\nc{\perf}{\mathsf{Perf}}
\nc{\ssets}{\mathsf{sSets}}
\nc{\cmpct}{\mathsf{cmpct}}
\nc{\compact}{\mathsf{cmpct}}
\nc{\pwrap}{\mathsf{PWrap}}
\nc{\coder}{\mathsf{Coder}}
\nc{\bimod}{\mathsf{Bimod}}
\nc{\grmod}{\mathsf{GrMod}}
\nc{\spaces}{\mathsf{Spaces}}
\nc{\pwrms}{\mathsf{PWrFuk}_{M,S}}
\nc{\pwrmf}{\mathsf{PWrFuk}_{M,F}}
\nc{\pwrapmf}{\mathsf{PWrFuk}_{M,F}}
\nc{\fuk}{\mathsf{Fukaya}}
\nc{\infwr}{\mathsf{InfWr}}
\nc{\fukaya}{\mathsf{Fukaya}}
\nc{\autml}{\mathsf{Aut}_{M,\Lambda}}
\nc{\fukml}{\mathsf{Fukaya}_{M,\Lambda}}
\nc{\fukmle}{\mathsf{Fukaya}_{M,\Lambda,\epsilon}}
\nc{\fukmod}{\wrfukcompact(M)\modules}
\nc{\lag}{\mathsf{Lag}}
\nc{\lagm}{\lag_M}
\nc{\lago}{\lag^o}
\nc{\lagml}{\lag_{M,\Lambda}} % For when I get lazy.
\nc{\lagmle}{\lag_{M,\Lambda,\epsilon}}
\nc{\fun}{\mathsf{Fun}}
\nc{\vect}{\mathsf{Vect}}
\nc{\chain}{\mathsf{Chain}}
\nc{\wrfuk}{\mathsf{WrFukaya}}
\nc{\wrfukcompact}{\mathsf{WrFukaya}_{\mathsf{cmpct}}}
\nc{\pwrfuk}{\mathsf{PWrFukaya}}
\nc{\inffuk}{\mathsf{InfFuk}}
\nc{\pwrfukml}{\mathsf{PWrFukaya}_{M,\Lambda}}
\nc{\inffukml}{\mathsf{InfFuk}_{M,\Lambda}}
\nc{\nattrans}{\mathsf{NatTrans}}
\nc{\corres}{\mathsf{Corres}}
\nc{\fukep}{\fukaya_\Lambda(M,\epsilon)}
\nc{\fukepop}{\fukaya_\Lambda(M,\epsilon)^{\op}}
\nc{\lagep}{\lag_\Lambda(M,\epsilon)}
\DMO{\cyl}{cyl} % Cylindrical
\nc{\dbcoh}{D^b\mathsf{Coh}}
\nc{\corr}{\mathsf{Corr}}

\nc{\cat}{\mathsf{Cat}}
\nc{\Cat}{\mathsf{Cat}}
\nc{\ainfty}{\mathsf{A}_\infty}
\nc{\inftycat}{\mathcal{C}\!\operatorname{at}_\infty}
\nc{\Ainftycat}{\mathcal{C}\!\operatorname{at}_{A_\infty}}
\nc{\ainftycat}{\mathcal{C}\!\operatorname{at}_{A_\infty}}
\nc{\stablecat}{\mathcal{C}\!\operatorname{at}_\infty^{\Ex}}

\DMO{\im}{im}
\DMO{\inj}{inj}
\DMO{\fib}{fib}
\DMO{\conf}{Conf}
\DMO{\chains}{Chains}
\DMO{\cochains}{Cochains}
\DMO{\cone}{Cone}
\DMO{\ran}{Ran}
\DMO{\rot}{Rot}
\DMO{\leg}{Leg}
\DMO{\imm}{imm}
\DMO{\adj}{adj}
\DMO{\cube}{Cube}
\DMO{\back}{back}
\DMO{\front}{front}
\DMO{\flow}{Flow}
\DMO{\floer}{Floer}
\DMO{\maps}{Maps}
\DMO{\exact}{exact}
\DMO{\Decomp}{Decomp}
\DMO{\decomp}{Decomp}
\DMO{\collar}{collar}
\DMO{\yoneda}{Yoneda}
\DMO{\hamspace}{Ham}
\DMO{\sympspace}{Symp}
\DMO{\holomaps}{Holomaps}
\DMO{\comp}{Comp}
\DMO{\crit}{Crit}
\DMO{\test}{{test}}
\DMO{\sign}{sign}
\DMO{\topp}{top}
\DMO{\indx}{Index}
\DMO{\Break}{Break} % Partitions
\DMO{\zero}{zero} %Zero
\DMO{\ob}{Ob}
\DMO{\gr}{Gr} % Grassmanian
\DMO{\Gr}{Gr} % Grassmanian
\DMO{\cl}{Cl} % Clifford Algebra
\DMO{\grlag}{GrLag}
\DMO{\Pin}{Pin}
\DMO{\Graph}{Graph}
\DMO{\pin}{Pin}
\DMO{\gap}{Gap}
\DMO{\Ex}{Ex}
\DMO{\id}{id}
\DMO{\End}{End}
\DMO{\sym}{Sym} 
\DMO{\aut}{Aut}
\DMO{\DK}{DK} %Dold-Kan
\DMO{\poly}{poly} % Polynomial deRham forms
\DMO{\diff}{Diff}
\DMO{\coll}{coll}
\DMO{\dist}{dist} %Distance function
\DMO{\coker}{coker} %Cokernel
\nc{\kernel}{\ker} %Kernel
\DMO{\sspan}{span}
\DMO{\hocolim}{hocolim}	
\DMO{\holim}{holim}
\DMO{\sk}{sk}

\DMO{\ho}{ho}
\DMO{\fin}{fin}
\DMO{\ret}{Ret}
\DMO{\ham}{Ham}
\DMO{\con}{con}
\DMO{\leaf}{leaf}
\DMO{\supp}{supp}
\DMO{\edge}{edge}
\DMO{\colim}{colim}
\DMO{\edges}{edges}
\DMO{\Image}{image}
\DMO{\roots}{roots}
\DMO{\height}{height}
\DMO{\finmod}{FinMod}
\DMO{\leaves}{leaves}
\DMO{\planar}{planar}
\DMO{\vertices}{vertices}

\nc{\lagg}{\lag^{\cG}}
\nc{\iso}{\mathsf{Iso}}
\nc{\Set}{\mathsf{Set}}
\nc{\ass}{\mathsf{ \bf Ass}}
\nc{\Mod}{\mathsf{Mod}}
\nc{\modules}{\mathsf{Mod}}
\nc{\sset}{\mathsf{sSet}}
\nc{\liou}{\mathsf{Liou}}
\nc{\poset}{\mathsf{Poset}}
\nc{\trno}{T^*\RR^n_{\geq 0}}
\nc{\spectra}{\mathsf{Spectra}}
\nc{\tensorfin}{\tensor^{\fin}}
\nc{\lagptg}{\lag_{pt,pt}^{\cG}}
\nc{\Fin}{\mathcal{F}\mathsf{in}}
\nc{\lagnl}{\lag_{N,\Lambda}}
\nc{\lagmlg}{\lag_{M,\Lambda}^{\cG}}
\nc{\lagsplit}{\lag^{\mathsf{split}}}
\nc{\lagktimes}{(\lag^{\dd k})^\times}
\nc{\lagplanar}{\lag^{\times,\planar}}

\nc{\smsh}{\wedge}
\nc{\un}{\underline}
\nc{\xto}{\xrightarrow}
\nc{\xra}{\xto}
\nc{\tensor}{\otimes}
\nc{\del}{\partial}
\nc{\dd}{\diamond}
\nc{\tri}{\triangle}
\nc{\bb}{\Box}
\nc{\into}{\hookrightarrow}
\nc{\onto}{\twoheadrightarrow}
\nc{\contains}{\supset}
\nc{\transverse}{\pitchfork}
\nc{\uncirc}{\underline{\circ}}
\nc{\Jbar}{\overline{J}}
\nc{\Fbar}{\overline{F}}
\nc{\delbar}{\overline{\del}}
\nc{\thetabar}{\overline{\theta}}
\nc{\omegabar}{\overline{\omega}}

\nc{\colldiff}{\diff^{\del}} 
\nc{\trbar}{\overline{T^*\RR}}
\nc{\tr}{T^*\RR}
\nc{\tsa}{Ts\cA}
\nc{\tsb}{Ts\cB}
\nc{\cmbar}{\overline{\cM}}
\nc{\crbar}{\overline{\cR}}

\nc{\vece}{ {\vec \epsilon}}	
\nc{\vecd}{ {\vec \delta}}
\nc{\ov}{\overline}
\DMO{\op}{op}
\nc{\opp}{ ^{\op}}

\nc{\hiro}{\textcolor{blue}}

\nc{\eqn}{\begin{equation}}
\nc{\eqnn}{\begin{equation}\nonumber}
\nc{\eqnd}{\end{equation}}
\nc{\enum}{\begin{enumerate}}
\nc{\enumd}{\end{enumerate}}

\def\cA{\mathcal A}\def\cB{\mathcal B}\def\cC{\mathcal C}\def\cD{\mathcal D}
\def\cE{\mathcal E}\def\cF{\mathcal F}\def\cG{\mathcal G}

\def\cM{\mathcal M}\def\cN{\mathcal N}\def\cP{\mathcal P}
\def\cR{\mathcal R}
\def\cW{\mathcal W}

\def\RR{\mathbb R}

\def\fC{\mathfrak C}

\title{Functors (between $\infty$-categories) that aren't strictly unital}
\author{Hiro Lee Tanaka}
\email{hirolee@math.harvard.edu}
\address{Department of Mathematics, Harvard University, Cambridge, MA}
\begin{document}

\maketitle

\begin{abstract}
Let $\cC$ and $\cD$ be quasi-categories (a.k.a. $\infty$-categories). Suppose also that one has an assignment sending commutative diagrams of $\cC$ to commutative diagrams of $\cD$ which respects face maps, but not necessarily degeneracy maps. (This is akin to having an assignment which respects all compositions, but may not send identity morphisms to identity morphisms.) When does this assignment give rise to an actual functor? We show that if the original assignment can be shown to respect identity morphisms up to homotopy, then there exists an honest functor of $\infty$-categories which respects the original assignments up to homotopy. Moreover, we prove that such honest functors can be chosen naturally with respect to the original assignments.
\end{abstract}

\tableofcontents

\section{Introduction}
$\infty$-categories, also known as quasi-categories~\cite{joyal} and weak Kan complexes~\cite{boardman-vogt}, are among the most useful models for $(\infty,1)$-categories. However, like any model, there is at least one shortcoming that can be a hindrance in practice. Here is one: Functors between $\infty$-categories must, {\em strictly} speaking, respect units. This short work presents a work-around.

Concretely: Let $\cC: \Delta^{\op} \to \sets$ be an $\infty$-category, and in particular, a simplicial set. Heuristically, the degeneracy map $s_0: \cC_0 \to \cC_1$ picks out a specific unit for each object of $\cC$. Moreover, a functor $F:\cC \to \cD$ between $\infty$-categories must be a natural transformation of simplicial sets, meaning that $F$ must in particular commute with all degeneracy maps on the nose:
	\eqnn
		F \circ s_i = s_i \circ F.
	\eqnd
This strictness presents issues both philosophical and practical. Philosophically, whether a map between algebras or categories respects units is a property, not additional structure---when there is a contractible space of units, we shouldn't have to say that some arbitrary choices of units are respected.

Practically, one would like to say that an algebra map or a functor that respects specified units {\em up to homotopy} should define a functor of $\infty$-categories without too much more work. Here are two examples in which one might encounter such an issue:
\enum
	\item
		When trying to translate one homotopical model into the model of $\infty$-categories. For example, if one fixes a base ring $k$, the notion of a unital $A_\infty$-category over $k$ does not require a specific choice of unit. However, one can construct a functor from the $A_\infty$-category of $A_\infty$-categories to the $\infty$-category of $\infty$-categories. This is called the $A_\infty$-nerve, or simplicial nerve construction for $A_\infty$-categories, and was developed by Faonte~\cite{faonte, faonte-2} and by the present author~\cite{tanaka-thesis}. An immediate annoyance arises: Which degeneracy maps do we choose for the nerve, when there is no specified unit for the $A_\infty$ category? The seasoned homotopy theorist knows in her heart of hearts that this should present no issue, as any two choices should be equivalent up to contractible choice. This instinct must be made precise. Also, as noted in~\cite{tanaka-thesis}, there are other work-arounds to this that do not involve the present work by using formal diffeomorphisms of an $A_\infty$-category, but such choices are cumbersome to make functorial for all $A_\infty$-categories at once.
	\item
		When trying to construct a functor between $\infty$-categories. Especially when one considers functors having geometric origins, it takes a lot of extra work to create a functor which respects units on the nose. Even worse---if we can construct a map that seems to respect some of the degeneracy maps, it may be the case that other degeneracy maps are not respected! We would rather be lazy about it, and say that if a putative functor ``$F: \cC \to \cD$'' respects all face maps and respects enough degeneracy maps, then we know that there is an honest map of simplicial sets $F' : \cC \to \cD$ somehow equivalent to $F$. This is the main motivation behind the present work. We use our result here to help prove the main theorem in~\cite{tanaka-pairing}, where we construct a pairing between an $\infty$-category of Lagrangian cobordisms of a Liouville domain with various versions of its Fukaya category.
\enumd

We now build up to our main result, Theorem~\ref{thm.localization}, which indeed allows us to infer the existence of a functor between $\infty$-categories whenever one has a putative functor which respects compositions and ``respects units enough.''

Let $\Delta$ be the category of finite, non-empty, linearly ordered sets. Consider the subcategory 
	\eqnn
	j: \Delta_{\inj} \into \Delta
	\eqnd
consisting of the same objects, but whose morphisms are the {\em injective} poset maps. Recall that a functor $\Delta_{\inj}^{\op} \to \sets$ is called a {\em semisimplicial set}. A map of semisimplicial sets is a natural transformation.

In what follows, given an $\infty$-category $\cC: \Delta^{\op} \to \sets$, we will let $j^*\cC = \cC_{\inj}$ denote the composition $\Delta^{\op}_{\inj} \to \Delta^{\op} \to \sets$. The functor $\cC \mapsto \cC_{\inj}$ has a left adjoint, and we consider the counit of the adjunction:
	\eqn\label{eqn.counit}
	\epsilon: \cC_+ \to \cC
	\eqnd
We give more details on $\cC_+$ and this counit map in Section~\ref{section.counit}. As we will explain there, $\cC_+$  contains a set of non-degenerate edges which is naturally identified with the set of degenerate edges of $\cC$---we call this collection of edges $s_0(\cC_0)$.

Our main technical result is:

\begin{theorem}\label{thm.localization}
The counit $\epsilon: \cC_+ \to \cC$ exhibits $\cC$ as a localization of $\cC_+$ along the edges $s_0(\cC_0)$ in $\cC_+$.
\end{theorem}

To prove the theorem, it suffices to prove it when $\cC$ is a $k$-simplex for $k\geq 0$. We prove the case $k=0$ in Section~\ref{section.k=0}, and the general case in Section~\ref{section.k>0}.

\begin{remark}
In general, $\cC_+$ is not an $\infty$-category (see Remark~\ref{remark.non-quasi}); regardless, it makes sense to speak of a localization of $\cC_+$---it is the localization of any fibrant replacement $\cC_+ \simeq \tilde \cC_+$ along the image of $s_0(\cC_0)$.
\end{remark}

\begin{remark}
The theorem further implies that the localization $\epsilon$ doesn't merely send $e \in s_0(\cC_0)$ to some invertible edge, it sends $e$ to an edge {\em homotopic to the identity} of an object. See Remark~\ref{remark.idempotent}.
\end{remark}
 
Let $\cC$ and $\cD$ be $\infty$-categories. Fix $F: \cC_{inj} \to \cD_{inj}$ a map of semisimplicial sets. By Theorem~\ref{thm.localization}, we obtain the following result:

\begin{theorem}\label{thm.semisimplicial}
Assume that for every vertex $x \in \cC_0$, $F(s_0 x)$ is an equivalence in $\cD$. Then there exists a functor $F': \cC \to \cD$ (i.e., a map of simplicial sets) such that for every simplex $Y$ of $\cC$, there is a homotopy in $\cD$ from $F(Y)$ to $F'(Y)$. Moreover, the data of such an $F'$ and of a coherent choice of such homotopies is unique up to contractible choice of natural equivalence.
\end{theorem}   

\begin{proof}[Proof of Theorem~\ref{thm.semisimplicial} assuming Theorem~\ref{thm.localization}]
By the Joyal model structure on simplicial sets, the counit map $\epsilon$ factors into an acyclic cofibration $\alpha: \cC_+ \simeq \tilde \cC_+$, followed by a fibration $\tilde \cC_+ \to \cC$, as depicted in the left half of the following diagram:
	\eqnn
		\xymatrix{
			\cC_+ \ar[d]_{\alpha}^{\sim} \ar[r]^{F_+} \ar@/_2em/[dd]_{\epsilon}
				& \cD \\
			\tilde \cC_+ \ar@{->>}[d] \ar[ur] \\
			\cC \ar@{-->}[uur]_{F'}
			&.
		}	
	\eqnd
Note that $\tilde \cC_+$ is a a quasi-category. (Here we are using the assumption that $\cC$ itself is a quasi-category.)

The semisimplicial map $F: \cC_{\inj} \to \cD_{\inj}$ defines a map of simplicial sets $F_+: \cC_+ \to \cD$ by adjunction. Since $\cC_+$ is cofibrant, $F_+$ may also be factored by the {\em same} acyclic cofibration, followed by a fibration to $\cD$, up to homotopy. The top, solid triangle in the diagram above depicts the corresponding 2-simplex in the $\infty$-category of $\infty$-categories; i.e., a 2-simplex exhibiting the homotopy from $F_+$ to the composite $\cC_+ \to \tilde \cC_+ \to \cD$. Note that given $F_+$, the space of such triangles is contractible---this is evidenced by examining the homotopy fibers above $F_+$ of the homotopy equivalence of Kan complexes
	\eqnn
		\alpha^*:
		\fun(\tilde \cC_+, \cD)^\sim
		\to
		\fun(\cC_+, \cD)^\sim.
	\eqnd
(Here, $X^\sim$ denotes the largest Kan complex contained in a simplicial set $X$. Recall that $\hom(-,-) := \fun(-,-)^\sim$ is the usual enrichment rendering Lurie's Cartesian model structure a simplicial model category, and that Lurie's Cartesian model structure for marked simplicial sets is equivalent to Joyal's model structure. We conclude that $\alpha^*$ is an equivalence as $\alpha$ is an equivalence in either model structure, both domains are cofibrant, and the codomain $\cD$ is fibrant.)

Now we seek a dashed map of simplicial sets $F'$ as indicated, along with a 2-simplex of quasi-categories rendering the righthand, bottom triangle a homotopy-commutative diagram. By Theorem~\ref{thm.localization} and the universal property of localization, we know that such a homotopy-commutative triangle with the dashed arrow exists if and only if the map $\tilde \cC_+ \to \cD$ satisfies the property that the edges $s_0(\cC_0) \subset (\cC_+)_1 \to \tilde \cC_+$ are sent to equivalences in $\cD$. This is the hypothesis of Theorem~\ref{thm.semisimplicial}. 

Finally, to see the uniqueness of the whole diagram, let $\fun_W(\tilde \cC_+, \cD)^\sim \subset \fun(\tilde \cC_+, \cD)^\sim$ be the inclusion of the connected components of those functors sending $s_0(\cC_0)$ to equivalences in $\cD$. By the universal property of localization, the pre-composition map
	\eqnn
		\fun(\cC,\cD)^\sim
		\to
		\fun_W(\tilde \cC_+, \cD)^\sim
	\eqnd
is a homotopy equivalence. This completes the proof.
\end{proof}

\begin{remark}
Note that a map of semisimplicial sets $F: \cC_{\inj} \to \cD_{\inj}$ is precisely an assignment which respects homotopy-coherent diagrams, but may not respect choices of units on the nose. Returning to our original motivation, the above theorem indeed says that such an assignment which {\em does} send degenerate edges to equivalences canonically determines a functor from $F': \cC \to \cD$, along with a homotopy from that functor to the original assignment $F$.
\end{remark}

Moreover, an application of this theorem tells us that the assignment $F \mapsto F'$ (which is well-defined up to contractible choice) can be made appropriately natural in $F$. 
Let $F : \cC_{\inj} \times \Delta^k_{\inj} \to \cD_{\inj}$ be a map of semisimplicial sets, and let $d_i F$, $0 \leq i \leq k$ be the restriction of $F$ to the various faces of $\cC_{\inj} \times \Delta^{k-1}_{\inj} \subset \cC_{\inj} \times \Delta^k_{\inj}$.

\begin{corollary}
Assume that for every vertex $x \in \cC_0 \times \Delta^k_0$, $F(s_0(x))$ is an equivalence in $\cD$. Then $F': \cC \times \Delta^k \to \cD$ induces a homotopy coherent diagram with faces $(d_i F)': \cC \times \Delta^{k-1} \to \cD$. In particular, when $k=1$, $F'$ induces a natural transformation from $(d_1 F)'$ to $(d_0 F)'$. 
\end{corollary}

\begin{proof}
Applying Theorem~\ref{thm.semisimplicial} to $F$, we know that given $F_+$ and $\epsilon$, the space of homotopy-commutative triangles
	\eqnn
		\xymatrix{
			(\cC \times \Delta^k)_+ \ar[r]^-{F_+}  \ar[d]_\epsilon & \cD \\
			\cC \times \Delta^k \ar[ur]_{F'}
		}
	\eqnd
is contractible. On the other hand, for any choice of $0 \leq i \leq k$, we have a strictly commuting square:
	\eqnn
		\xymatrix{
			(\cC \times \Delta^{k-1})_+ \ar[rr]^{(\id_\cC \times d_i)_+} \ar[d]_\epsilon && (\cC \times \Delta^k)_+ \ar[d]^\epsilon \\
			\cC \times \Delta^{k-1} \ar[rr]^{\id_\cC \times d_i} && \cC \times \Delta^k
		}
	\eqnd
where $d_i$ is the inclusion of the $i$th face. Gluing along the righthand copy of $\epsilon$, or by applying the theorem to $d_i F$, we have two homotopy-commutative triangles
	\eqnn
		\xymatrix{
		(\cC \times \Delta^{k-1})_+ \ar[rr]^{d_i F_+} \ar[d]_\epsilon && \cD \\
		\cC \times \Delta^{k-1} \ar[urr]_{d_i (F')} 
		}
		\qquad
		\xymatrix{
		(\cC \times \Delta^{k-1})_+ \ar[rr]^{d_i F_+} \ar[d]_\epsilon && \cD \\
		\cC \times \Delta^{k-1} \ar[urr]_{(d_i F)'} 
		}
	\eqnd
Again by the universal property of localizations, we conclude that there is a contractible space of natural equivalence between any choice of $(d_i F)'$, and $d_i (F')$; gluing these natural equivalences to the faces of $F'$, we obtain a $k$-simplex whose faces are given by $(d_i F)'$.
\end{proof}

\begin{remark}
Shortly after our work, Steimle~\cite{steimle} proved a result that is a ``producing $\infty$-categories'' analogue of our ``producing maps between $\infty$-categories'' result. Specifically, Steimle shows that if a semisimplicial set $\underline{\cC}$ satisfies the weak Kan condition, and if every vertex of $\underline{\cC}$ admits an idempotent self-equivalence, then there exists a simplicial set $\cC$ such that $\cC_{inj} = \underline{\cC}$; the degeneracy $s_0: \cC_0 \to \cC_1$ can be any assignment picking out idempotent self-equivalences. Moreover, $\cC$ will automatically be a quasi-category, and any two choices of $\cC$ are equivalent as quasi-categories. (This is in line with our initial philosophy of unitality being a property, not a specified structure.)

By putting our results together, one can significantly simplify the efforts of constructing quasi-categories, and of constructing maps between them. Given a semisimplicial map $F: \underline{\cC} \to \underline{\cD}$, all one needs to verify is that, for every object $X$, some idempotent self-equivalence of $X$ is sent to some idempotent self-equivalence of $F(X)$. Then one naturally deduces the existence of a functor $\cC \to \cD$. Note this is a valid strategy even if the original $F$ does not respect some other choice of degeneracy maps for $\cC$ and $\cD$.
\end{remark}

{\bf Acknowledgments.} 
We thank Jacob Lurie for suggesting the proof method presented here, which simplifies (and makes functorial) an earlier method we had: One can construct a coCartesian fibration over $\Delta^1$ having $\cC$ and $\cD$ as fibers. We don't display that method here because of the obvious advantages of the present method. We also thank Clark Barwick for catching a mistake in an earlier draft, and to the editors and anonymous referees for very helpful comments. This work was conducted while I was supported by the National Science Foundation under Award No. DMS-1400761.

\section{Preliminaries}
Below, when we say {\em category}, we mean category in the usual sense of Eilenberg and Mac Lane~\cite{eilenberg-maclane}, and when we say $\infty$-category, a simplicial set satisfying the weak Kan condition. 

\begin{remark}
To make it easier for the reader to look up references, we remark on the history of the term ``semisimplicial set.'' For some time, the notion of a simplicial set was called a ``complete semi-simplicial complex,'' and what we call a semisimplicial set was called a ``semi-simplicial complex.'' This is due, as far as we know, to the original definitions in~\cite{eilenberg-zilber}. Our use of the prefix semi follows~\cite{htt}, and refers to the absence of identities/degeneracies; the original use of the prefix~\cite{eilenberg-zilber} most likely referred to the fact that a simplex is not determined by its vertices.
\end{remark}

\subsection{The counit of the adjunction and the edges $s_0(\cC_0)$}\label{section.counit}
We let $j:\Delta^{\op}_{\inj} \to \Delta^{\op}$ denote the inclusion. Note that there is an adjunction
	\eqnn
		\xymatrix{
			\fun(\Delta^{\op}_{\inj} , \sets)			\ar@/^/[r]^{j_!}
			& \ar@/^/[l]^{j^*} \fun(\Delta^{\op},\sets)
		}
	\eqnd
where $j^*$ is pre-composition with $j$ and $j_!$ is left Kan extension. If $\cC$ is a quasi-category, let $\cC_+$ denote the simplicial set obtained by first restricting $\cC$ along $\Delta^{\inj}$, then left Kan extending:
	\eqnn
		\xymatrix{
			\Delta^{\op} \ar@{-->}[r]^{\cC_+} & \sets \\
			\Delta_{\inj}^{\op} \ar[u] \ar[ur]_{\cC_{\inj}}.
		}
	\eqnd
It exists since $\sets$ has all small colimits, and by the universal property of left Kan extensions, there is a natural map (the counit of the adjunction):
	\eqn
		\cC_+ \to \cC.
	\eqnd

\begin{lemma}[Description of $\cC_+$]
The set of $N$-simplices of $\cC_+$ is given by the formula
	\eqn\label{eqn.C+}
		(\cC_+)_N
		\cong
		\coprod_{s: [N] \onto [N']} \cC_{N'}.
	\eqnd
That is, the $N$-simplices of $\cC_+$ contain a copy of $\cC_{N'}$ for each surjection $[N] \onto [N']$. The non-degenerate simplices of $\cC_+$ are precisely those for which the surjection $s: [N] \onto [N']$ is a bijection (i.e., the identity morphism).
\end{lemma}

\begin{proof}
By the usual formula for left Kan extension, the set $(\cC_+)_N$ is computed as the colimit of $\cC_{\inj}$ over the diagram category $(\Delta^{\op}_{\inj})_{/ [N]}$. This diagram category has a final, discrete subcategory of all objects coming from surjective maps $s: [N] \onto [N']$ in $\Delta$. This gives \eqref{eqn.C+}.
\end{proof}

\begin{remark}
Equation~\eqref{eqn.C+} agrees with the intuition that $\cC_+$ freely adjoins degeneracies to $\cC$.
\end{remark}

\begin{remark}[The simplicial maps of $\cC_+$]
Now assume we are given an element of $(\cC_+)_N$, which we think of as a pair $(S,s)$ with $S \in \cC_{N'}$ and $s: [N] \onto [N']$. Given any poset map $f: [M] \to [N]$, the induced function $(\cC_+)_N \to (\cC_+)_M$ is as follows: there is a unique factorization of $s \circ f$
	\eqnn
		\xymatrix{
			[M] \ar[r]^f  \ar@{->>}[d]^{s'} & [N] \ar@{->>}[d]^{s} \\
			[M'] = \im(s \circ f) \ar[r]^-{f'} & [N']
		}
	\eqnd
where $s'$ is a surjection. The map $f'$ is an injection, so $\cC_{inj}$ tells us how to transport an element of $\cC_{N'}$ to an element of $\cC_{M'}$, and we consider this as the $s'$ component of $(\cC_+)_M$. This completes our description of $\cC_+$.
\end{remark}

\begin{remark}
We can explicitly describe the simplicial set maps associated to $\Delta^k_+$. That is, for any poset map
$g: [m] \to [\un{m}]$, we describe the induced map
	\eqnn
		(f,s) \mapsto (g^*f, g^*s),
		\qquad
		(\Delta^k_+)_{m} \to (\Delta^k_+)_{\un{m}}.
	\eqnd
We note that there is a unique surjection $\un{s}$, a unique injection, and a unique map $\un{f}$ fitting into the diagram below:
	\eqnn
		\xymatrix{
			[k] 
				& [m'] \ar[l]_{f}
				& [m] \ar@{->>}[l]_{s}  \\
			[k]\ar@{=}[u]
				& [\un{m}'] \ar[l]_{\un{f}} \ar@{^{(}->}[u]
				& [\un{m}] \ar@{->>}[l]_{\un{s}} \ar[u]^g
		}
	\eqnd
It is obtained by setting $[m'] \cong \im(s \circ g)$. We let
	\eqnn
		g^*(f,s) := (g^*f, g^*s) := (\un{f},\un{s}).
	\eqnd
\end{remark}

\begin{remark}\label{remark.non-quasi}
The simplicial set $\cC_+$ is almost never a quasi-category. While any inner 2-horn is fillable, not every 3-horn is. For example, consider an inner horn $\Lambda^3_1 \to \cC_+$ with subsimplices we label as follows:
\enum
\item Vertices $x_0,\ldots, x_3$,
\item Edges $f_{ij} : x_i \to x_j$, and
\item 2-simplices $T_{012}$, $T_{013}$, $T_{123}$. We further assume
\item $T_{123} = s_1^{\cC_+}(f_{12})$. (This necessarily implies that $f_{12} = f_{13}$, and that $f_{23} = s_0^{\cC_+}(x_2)$.)
\item $T_{012}$ and $T_{013}$ are 2-simplices in $\cC$, and
\item $f_{02} \neq f_{03}$. (This implies that $\cC$ is not the nerve of a category.)
\enumd
There is no 3-simplex filling such a horn because there is no 2-simplex in $\cC_+$ with boundary conditions given by edges $f_{02}$, $f_{03}$, and $s_0^{\cC_+}(x_2)$ while $f_{02} \neq f_{03}$.

However, if the original simplicial set is isomorphic to the nerve of some small category, $\cC_+$ is a quasi-category. To see this, one can observe that any sequence of $k$ composable edges has a unique $k$-simplex with those edges as boundary edges.

\end{remark}

Finally, let us explain the set $s_0(\cC_0)$ along which we must localize. By using the Joyal model structure, the counit map~\eqref{eqn.counit} factors as 
	\eqnn
		\cC_+ \to \tilde \cC_+ \to \cC
	\eqnd
where the lefthand arrow is an acyclic cofibration, and the righthand arrow to $\cC$ is an inner fibration. Since we have assumed $\cC$ is a fibrant object (i.e., a quasi-category), $\tilde \cC_+$ is a fibrant replacement for $\cC_+$. It models the free $\infty$-category generated by the simplicial set $\cC_+$.

As such, $\cC_+$ has a set of 1-simplices identified with a copy of $\cC_1$, and a copy of $\cC_0$. The former of course contains the degenerate edges $s_0^{\cC}(\cC_0)$; these are the edges $s_0(\cC_0)$ of Theorem~\ref{thm.localization}.

\subsection{Computing localizations as pushouts}\label{section.pushouts}
Whenever constructing the localization $\cE[W^{-1}]$ with respect to a subsimplicial set $W \subset \cE$, we can compute the localization as the (honest) pushout of the diagram
	\eqn\label{eqn.pushout}
		\xymatrix{
			W \ar[r] \ar[d] & \cE \\
			|W|
		}
	\eqnd
where $|W|$ is a Kan replacement of $W$. Note that an honest pushout computes the homotopy pushout because the Joyal model structure is left proper, and the arrow $W \into \cE$ is a cofibration.

\begin{remark}
In our setting $\cE = \cC_+$ may not be fibrant in the Joyal model structure, nor the pushout $\cP$.
\end{remark}

\subsection{Necklaces to compute mapping spaces}\label{section.necklaces}
Using the notation from Section~1.1.5 of~\cite{htt}, recall the adjunction
	\eqnn
		\xymatrix{
			\ssets			\ar@/^/[r]^{\fC}
			& \ar@/^/[l]^{N} \Cat_{\Delta}.
		}
	\eqnd
The right adjoint is the simplicial nerve functors $N$, and $\fC$ is the left adjoint; it takes any simplicial set to a category enriched in simplicial sets. Moreover, weak equivalences in the Joyal model structure are those simplicial set maps $f: X \to Y$ which become equivalences after applying $\fC$. 

While an explicit description of $\fC$ would be nice, it is in general computed as a colimit of simplicially enriched categories, so it is not so easy to write down. 

However, one can give a name to the mapping spaces of $\fC$ by the work of Dugger-Spivak~\cite{dugger-spivak}, which we now review.

\begin{defn}[Necklaces]
Let $m_0,\ldots,m_a$ be a sequence of integers $\geq 0$. A necklace $T$ is a simplicial set of the form
	\eqnn
		T = \Delta^{m_0} \vee \ldots \vee \Delta^{m_a}
	\eqnd
where the final vertex of $\Delta^{m_i}$ is attached to the initial vertex of $\Delta^{m_{i+1}}$. 
A map of necklaces is a map of simplicial sets $T \to T'$ respecting the initial vertex of $\Delta^{m_0}$ and the final vertex of $\Delta^{m_a}$.
\end{defn}

Now let $X$ be any simplicial set. Fix two vertices $x,y \in X$. Then we define the category
	\eqnn
	(Nec \downarrow X)_{x,y}
	\eqnd
(in the usual sense of Eilenberg and Mac Lane) to be a category where
\begin{itemize}
	\item
		Objects are maps
	\eqnn
		T \to X
	\eqnd
where $T$ is a necklace, the initial vertex of $T$ is sent to $x$, and the terminal vertex of $T$ is sent to $y$. 
	\item
		Morphisms are commutative diagrams 
			\eqnn
				\xymatrix{
					T \ar[rr] \ar[dr] && T' \ar[dl] \\
					& X
				}
			\eqnd
	where $T \to T'$ is a map of necklaces.
\end{itemize}

We let $N( Nec \downarrow X)_{x,y}$ denote the nerve of this category. It is a simplicial set, as usual, whose simplices are commutative diagrams in the shape of a simplex.

\begin{theorem}[\cite{dugger-spivak}, Theorem~5.2]\label{thm.DS}
The simplicial sets
	\eqnn
		\hom_{\fC[X]}(x,y)
		\qquad
		\text{and}
		\qquad
		N( Nec \downarrow X)_{x,y}
	\eqnd
are weakly homotopy equivalent.
\end{theorem}

We will use this theorem in computing the mapping spaces of the localization when $\cC = \Delta^k$ in Section~\ref{section.k>0} below.

\subsection{Finality}
Let $Y$ be an $\infty$-category. We say a map of simplicial sets $f: X \to Y$ is {\em final} if and only if, for every object $y \in Y$, the fiber product simplicial set $X \times_Y Y_{y/}$ is weakly contractible. What we call final in this paper is called cofinal in~\cite{htt}. It is a result due to Joyal---also stated in Theorem~4.1.3.1 of~\cite{htt}---that our definition of final is equivalent to other standard definitions (which sometimes refer to our notion as ``cofinal'').

\begin{example}\label{example.final-semi}
The inclusion $N(\Delta_{inj}) \into N(\Delta)$ is final. (See Lemma~6.5.3.7 of~\cite{htt}.) This is is an oft-used fact when computing geometric realizations of simplicial sets---cofinal maps induce equivalences of colimits, and this finality is one way to see that the geometric realization is insensitive to the degenerate simplices. 
\end{example}

In this paper, the main utility of final maps will be the following:

\begin{lemma}[Proposition 4.1.1.3 of~\cite{htt}]\label{lemma.final-equivalence}
A final map $X \to Y$ induces a homotopy equivalence of geometric realizations $|X| \simeq |Y|$.
\end{lemma}

\section{Proof of Theorem~\ref{thm.localization} when $\cC = \Delta^0$}\label{section.k=0}
As mentioned before, to prove Theorem~\ref{thm.localization}, it suffices to prove the claim when $\cC$ is a $k$-simplex. We begin when $k=0$:

\begin{lemma}\label{lemma.0-simplex} 
The localization of $\Delta^+_0$ is equivalent to $|\Delta^+_0|$, and the latter is contractible. Hence Theorem~\ref{thm.localization} is true when $\cC = \Delta^0$.
\end{lemma}

\begin{proof}[Proof of Lemma~\ref{lemma.0-simplex}.]
By~\eqref{eqn.C+}, $\cC_+$ is a simplicial set with exactly one non-degenerate simplex in each dimension. (It is in fact isomorphic to the simplicial set $\idem$ from~\cite{htt}, Section~4.4.5.) Then $s_0(\cC_0)$ is the single non-degenerate 1-simplex, so the localization $\cC_+[s_0(\cC_0)^{-1}]$ is just the Kan completion $|\cC_+|$ of $\cC_+$. (Equivalently, it is the fibrant replacement of $\cC_+$ in the Quillen model structure for simplicial sets.)

One explicitly computes the homology groups of the localization by using the normalized chain complex for a simplicial set. Since homology is preserved by fibrant replacement, one easily sees that the homology groups $H_k(|\cC_+|)$ vanish for $k > 0$. 

$\pi_1$ is likewise easily computed: If $e$ is the unique non-degenerate edge, the unique non-degenerate 2-simplex realizes the relation $e \circ e \sim e$ in $\cC_+$. Hence the element of $\pi_1$ generated by $e$ is an idempotent---since $\pi_1$ is a group, this means $[e]$ must be equal to the identity. Thus $\pi_1$ is also a trivial group.

By Whitehead's theorem, the localization is thus a contractible Kan complex, hence categorically equivalent to $\Delta^0$.
\end{proof}

\begin{remark}\label{remark.idempotent}
There is another proof using formal properties of adjunctions: Since we know that the localization is the Kan completion of $\Delta^0_+$, it suffices to prove that the counit map $j_! j^* \Delta^0 \to \Delta^0$ is an equivalence in the Quillen model structure (where fibrant objects are Kan complexes). This is equivalent to the unit map $\Delta^0 \to j^*j_!\Delta^0$ being an equivalence by Zorro's Lemma (a.k.a. the triangular identities, a.k.a. the zigzag identities) for the co/units of an adjunction. We are done because the unit map is an equivalence in the Quillen model structure for any simplicial set.

Regardless, the alternative proof we presented illustrates why localizing with respect to $e \in s_0\cC$ sends $e$ to something {\em homotopic to the identity}, rather than to some other invertible map. The relation $e \circ e \sim e$ forces $e$ to be homotopic to the identity by composing with $e^{-1}$ in the localization.
\end{remark}

\section{Proof of Theorem~\ref{thm.localization} when $\cC = \Delta^k$ with $k \geq 1$}\label{section.k>0}
Now we wish to prove Theorem~\ref{thm.localization} when $\cC = \Delta^k$ for $k \geq 1$, so we must consider the localization of $\cC_+$ along the simplicial subset $W = \coprod_{i \in [k]} \Delta^0_+$. Now, the case $k=0$ shows that $|W| \simeq \coprod_{i \in [k]} \Delta^0$. Replacing $|W|$ accordingly, the localization is modeled by the honest pushout
	\eqn\label{eqn.pushout-simplex}
		\xymatrix{
			W \ar[r] \ar[d] & \cC_+ \ar[d] \\
			\coprod_{i \in [k]} \ast \ar[r] & \cP
		}
	\eqnd
as we recalled in Section~\ref{section.pushouts}. Because we wish to prove that $\cP$ is equivalent to $\Delta^k$ as a quasi-category, we now proceed to show that all of its hom spaces are contractible. 

To do this, we utilize Theorem~\ref{thm.DS}. Using the notation from Section~\ref{section.necklaces}, our goal is to consider $X = \cP$ and show that the simplicial sets $N( Nec \downarrow X)_{x,y}$ are contractible for any choice of vertices $x,y$. We define some auxiliary categories.

\begin{defn}
Let $\cN_{x,y} \subset (Nec\downarrow X)_{x,y}$ denote the full subcategory of those $T \to X$ such that for each $i$, the composition $\Delta^{m_i} \subset T \to X$ is a non-degenerate simplex in $X$.

Let $\cF_{x,y} \subset \cN_{x,y}$ denote the full subcategory of those $T \to X$ where $T \cong \Delta^m$, and where $T$ hits every vertex between $x$ and $y$, inclusive.
\end{defn}

\begin{lemma}\label{lemma.cofinal}
The inclusion of nerves
	\eqnn
		N(\cN_{x,y}) \to N(Nec \downarrow X)_{x,y}
	\eqnd
is a final map. When $X = \cP$ as in \eqref{eqn.pushout-simplex}, the inclusion of nerves
	\eqnn
		N(\cF)_{x,y} \to N(\cN_{x,y}) 
	\eqnd
is also final.
\end{lemma}

\begin{proof}[Proof of Lemma~\ref{lemma.cofinal}.]
Without loss of generality, we may assume $x=0$ and $y=k$. We will drop the subscripts $x,y$ from the notation out of sloth.

We first prove the first arrow is final. We do this by showing that the fiber product category
	\eqnn
		\cN \times_{(Nec \downarrow X)} (Nec \downarrow X)_{T/}
	\eqnd
has an initial object for each object $T \in \cN$. Since the nerve of any category with an initial object is contractible, this proves finality. But the existence of an initial object is essentially a consequence of Proposition~4.7 of~\cite{dugger-spivak}. There it is proved that for any $T$, there is an initial $T'$ such that the arrow $T \to T'$ is a surjection of necklaces. Then any other object $(g: T \to T'')$ in the fiber product category receives a unique map from $T'$---the map picks out the image of $g$.

To prove the second assertion, we prove that the fiber product category 
	\eqnn
		\cF \times_{\cN} \cN_{T/}
	\eqnd
has a terminal object for any $T$. So let $T \to \cP$ be a necklace each of whose simplices is non-degenerate. We ignore its 0-simplices without loss of generality, so that $T \cong \Delta^{m_0} \vee \ldots \vee \Delta^{m_a}$ with each $m_i > 0$. By definition of $\cP$, each non-degenerate $m$-simplex with $m>0$ must hit more than one vertex of $\Delta^k$---such non-degenerate $m$-simplices of $\cP$ are in bijection with such non-degenerate simplices in $\Delta^k_+$. Further, non-degenerate simplices in $\Delta^k_+$ are in bijection with (arbitrary) simplices in $\Delta^k$ by \eqref{eqn.C+}. Since $\Delta^k$ is the nerve of a category, inner horns have unique fillers, and hence necklaces uniquely fill to a simplex of dimension $M = m_0 + \ldots + m_a$. In other words, we have found a unique filler
	\eqnn
		\xymatrix{
			\Delta^{m_0} \vee \ldots \vee \Delta^{m_a} \ar[d] \ar[r] & \Delta^k \\
			\Delta^M \ar@{-->}[ur] & .
		}
	\eqnd
In $\Delta^k$, such a simplex corresponds to a poset map $j: [M] \to [k]$, though not necessarily surjective (and hence not an object in $\cF$). However, there is a natural completion of $M$ to a surjection---endow the set
	\eqnn
	[M'] \cong [M] \bigcup \left( [k]\setminus j([M]) \right)
	\eqnd
with the linear ordering so that $[M'] \to [k]$ is a surjective poset map and so that the composite $[M] \to [M'] \to [k]$ is equal to $j$. It is easy to check that this $[M']$ is a terminal object.
\end{proof}

\begin{lemma}\label{lemma.NcF}
$|N(\cF)_{x,y}|$ is contractible.
\end{lemma}

\begin{proof}[Proof of Lemma~\ref{lemma.NcF}.]
Without loss of generality, we may assume that $x=0$ and $y=k$. This means $\cF$ can be described as follows: 
An object is a choice of non-degenerate $m$-simplex in $\Delta^k_+$
		\eqnn
			\Delta^m \to \Delta^k_+
		\eqnd
hitting every vertex from 0 to $k$. In terms of \eqref{eqn.C+}, this is exactly the data of a pair 	
		\eqnn
			(f,s=\id)
		\eqnd
where $f: [m] \to [k]$ is a surjection, and $s: [m] \cong [m]$ is the identity map. Now, since $\cF$ is the full subcategory where $(f,s) = (f,\id)$, a morphism is given by any $g$ such that the image map $[\un{m}] \to \im(g)$ is an isomorphism---i.e., $g$ must be an injection. 

To summarize, $\cF$ is equivalent to the category whose objects are surjective poset maps $f: [m] \to [k]$, and whose morphisms are injective poset maps $g: [\un{m}] \to [m]$ satisfying $f \circ g = \un{f}$. By identifying
	\eqnn
		(f: [m] \to [k])
		\mapsto
		([m_0],\ldots,[m_k]),
		\qquad
		[m_i] \cong f^{-1}(i),
	\eqnd
one obtains an isomorphism of categories $\cF \cong (\Delta_{inj})^{1+k} = \Delta_{\inj} \times \ldots \times \Delta_{\inj}$. 

So we are finished if we can prove that the nerve $N(\Delta_{inj})$ has contractible realization. This can be done in any number of ways: For instance, as mentioned in Example~\ref{example.final-semi}, the inclusion $N(\Delta_{inj}) \into N(\Delta)$ is final. On the other hand, $|N(\Delta)|$ is contractible because $\Delta$ has a terminal object called $[0]$. Since final maps induce homotopy equivalences of realizations, $|N(\Delta_{inj})|$ is contractible.

This completes the proof.
\end{proof}

\begin{lemma}\label{lemma.contractible}
Let $\cP$ be the simplicial set from \eqref{eqn.pushout-simplex}. Then all the mapping simplicial sets of $\fC[\cP]$ are weakly contractible.
\end{lemma}

\begin{proof}[Proof of Lemma~\ref{lemma.contractible}.]
We drop the subscripts $x,y$ out of sloth.
Passing to Kan completions/geometric realization of simplicial sets, we have a sequence of maps
	\eqnn
		\xymatrix{
		\ast \ar[rr]_-{\sim}^-{\text{Lem~\ref{lemma.NcF}}}
			&& |N(\cF)| \to 
			|N(\cN) | \to 
			|N( Nec \downarrow X) |\ar[rr]^-{\text{Thm~\ref{thm.DS}}}_-{\sim}
			&&|\hom_{\fC[X]}|.
		}
	\eqnd 
The unlabeled arrows are homotopy equivalences because they come from final maps by Lemma~\ref{lemma.cofinal}, and final maps induce homotopy equivalences of realizations. This completes the proof.
\end{proof}

\begin{remark}
We can give some intuition as to why  Lemma~\ref{lemma.contractible} should be true: $\cP$ has an underlying $\infty$-groupoid equal to $\coprod \Delta^0$. We note (1) the localization does not invert any morphisms which are not endomorphisms, and (2) $\cC_+$ is the free $\infty$-category generated from $\coprod \idem$ (see the proof of Lemma~\ref{lemma.0-simplex}) by adjoining a unique morphism from the $i$th object to the $j$th object for $i<j$. This means that $\hom_\cP(i,j)$ is homotopy equivalent to the free bimodule for $\End_\cP(i) \times \End_\cP(j) \cong \Delta^0 \times \Delta^0$---i.e., a contractible space.
\end{remark}

\begin{corollary}
Theorem~\ref{thm.localization} is true when $\cC = \Delta^k$ for $k \geq 1$.
\end{corollary}

\begin{proof}
It suffices to show that the natural map $\fC[\cP] \to \fC[\Delta^k]$ is a weak equivalence of categories enriched in simplicial sets.  Note that the map obviously induces a bijection on objects. Also obvious is that the hom simplicial sets of the codomain are all contractible. Finally, Lemma~\ref{lemma.contractible} says that all the hom simplicial sets of the domain are contractible, so the natural map is an equivalence.
\end{proof}

\bibliographystyle{amsalpha}
\bibliography{biblio}

%\layout

%How To Compile
%
%The document is test.tex and I compiled it as follows.
%
%xelatex test
%makeindex test
%makeindex test.nlo -s nomencl.ist -o test.nls
%xelatex test

%	How to make new notation index
%   \newnotation{NOTATION}{EXPLANATION}
%	\index{Name}

\end{document}